\documentclass[11pt]{amsart}
\usepackage[top = 0.8in, bottom = 0.8in, left = 1.1in, right=1.1in,
marginparwidth=1.3in]{geometry}
\usepackage{setspace}
\usepackage{amsmath, amsfonts, amsthm, amstext, xspace}
\usepackage[mathscr]{euscript}
\usepackage{amssymb,latexsym}
\usepackage{comment}
\usepackage{xcolor}
\usepackage{hyperref}
\usepackage[capitalise]{cleveref}
\definecolor{seagreen}{RGB}{46,139,87}
\definecolor{maroon}{RGB}{128,0,0}
\definecolor{darkviolet}{RGB}{148,0,211}
\definecolor{twelve}{RGB}{100,100,170}
\definecolor{thirteen}{RGB}{100,150,50}
\definecolor{fourteen}{RGB}{200,0,0}
\definecolor{fifteen}{RGB}{0,200,0}
\definecolor{sixteen}{RGB}{0,0,200}
\definecolor{seventeen}{RGB}{200,0,200}
\definecolor{eighteen}{RGB}{0,200,200}

\usepackage{centernot}
\usepackage{longtable}
\usepackage{mathtools}
\usepackage{stmaryrd}
\makeatletter
\newcommand{\xMapsto}[2][]{\ext@arrow 0599{\Mapstofill@}{#1}{#2}}
\def\Mapstofill@{\arrowfill@{\Mapstochar\Relbar}\Relbar\Rightarrow}
\makeatother

\usepackage[all]{xy}
\newtheorem{thm}{Theorem}[section]
\newtheorem*{theorem*}{Theorem}
\newtheorem*{conjecture*}{Conjecture}
\newtheorem*{corollary*}{Corollary}
\newtheorem{lemma}[thm]{Lemma}
\newtheorem{problem}[thm]{Problem}
\newtheorem{theorem}[thm]{Theorem}
\newtheorem{corollary}[thm]{Corollary}

\newtheorem{proposition}[thm]{Proposition}
\newtheorem{conjecture}[thm]{Conjecture}

\newtheorem{convention}[thm]{Convention}

\theoremstyle{definition}
\newtheorem*{question*}{Question}
\newtheorem{definition}[thm]{Definition}

\newtheorem{remark}[thm]{Remark}

\newtheorem{question}[thm]{Question}
\newtheorem{construction}[thm]{Construction}

\newtheorem{thmx}{Theorem}

\numberwithin{equation}{section} 

\def\c{\mathbb{C}}

\def\h{\mathbb{H}}

\def\k{\mathbb{K}}

\def\r{\mathbb{R}}

\def\t{\mathbb{T}}
\def\z{\mathbb{Z}}

\def\cp{\mathcal{P}}

\def\Ext{\operatorname{Ext}}

\def\colim{\operatorname{colim}}
\def\cofib{\operatorname{cofib}}

\def\coker{\operatorname{coker}}

\def\im{\operatorname{im}}

\def\tr{\operatorname{tr}}

\def\pr{\operatorname{pr}}
\def\fr{\operatorname{fr}}
\def\Th{\operatorname{Th}}

\definecolor{llteal}{RGB}{198,232,227}
\definecolor{llred}{RGB}{237,228,228}
\definecolor{llgray}{RGB}{230,230,230}
\definecolor{maroon}{RGB}{150,0,0}
\definecolor{orange}{RGB}{255,165,0}

\newcommand{\highlight}[1]{\ifmmode{\text{\sethlcolor{llgray}\hl{$#1$}}}\else{\sethlcolor{llred}\hl{#1}}\fi}

\usepackage{amssymb}
\usepackage{tikz-cd}

\author{Boris Botvinnik}\address{University of
  Oregon}\email{botvinn@uoregon.edu}
\author{J.D. Quigley}\address{University of
  Virginia}\email{mbp6pj@virginia.edu}
\title[Symmetries of exotic
  spheres and Mahowald invariants]{Symmetries of exotic
  spheres via complex and quaternionic Mahowald invariants}

\begin{document}
\maketitle

\begin{abstract}
 We use new homotopy-theoretic tools to prove the
    existence of smooth $U(1)$- and $Sp(1)$-actions on infinite
    families of exotic spheres. Such families of spheres are
    propagated by the complex and quaternionic analogues of the
    Mahowald invariant (also known as the root invariant).  In
  particular, we prove that the complex (respectively, quaternionic)
  Mahowald invariant takes an element of the $k$-th stable stem $\pi_k^s$ represented by a
  homotopy sphere $\Sigma^k$ to an element of a higher stable stem $\pi_{k+\ell}^s$
  represented by another homotopy sphere $\Sigma^{k+\ell}$ equipped
  with a smooth $U(1)$- (respectively, $Sp(1)$-) action with fixed
  points the original homotopy sphere $\Sigma^k\subset
  \Sigma^{k+\ell}$.
%
%
%
%
\end{abstract}

\tableofcontents

\vspace*{-10mm}

\section{Introduction}

The $n$-sphere $S^n$ equipped with its standard smooth structure is
distinguished among smooth $n$-manifolds. From the perspective of
smooth transformation groups, $S^n \subseteq \r^{n+1}$ admits a smooth
nontrivial $SO(n+1)$-action, making it the most symmetric of all
smooth simply-connected $n$-manifolds. From the perspective of
Riemannian geometry, the $n$-sphere of radius $R$ is particularly
nice, with constant sectional curvature $1/R^2$, Ricci curvature
$(n-1)/R^2$, and scalar curvature $n(n-1)/R^2$.

On the other hand, there are exotic spheres, smooth manifolds which are homeomorphic, but not diffeomorphic, to $S^n$. Decades of effort in geometric topology and stable
homotopy theory (e.g., \cite{Mil56, KM63, HHR16, WX17, BHHM20, BMQ22}) have
shown that many exotic spheres exist. However, very
little is known about their smooth transformation groups and
curvature.

On the transformation group side, upper bounds on their degrees of
symmetry (the maximal dimension of a compact Lie group acting smoothly and effectively on them) have appeared in \cite{Hsi67, HH67, HH69, LY74, Str94}. In
cases where exotic spheres can be described explicitly, e.g., the
Brieskorn representations of the Kervaire spheres, it is possible to
exhibit such actions of high-dimensional compact Lie groups.
However, for an arbitrary exotic sphere, lower bounds are
harder to come by, cf. \cite{Bre67, Jos81, Sch72, Sch75, Sto88,
  Qui22}. Schultz highlighted the following question in 1985:

\begin{question*}[{Schultz, \cite{Sch85}}]\label{Ques:TAction}
Let $\Sigma^n$ be an exotic $n$-sphere, $n \geq 5$. Does $\Sigma^n$
support a nontrivial smooth $U(1)$-action?
\end{question*}

In \cite{Qui22}, the second author applied a result of Schultz
\cite{Sch85} to provide some positive answers to this question
using stable homotopy theory. In this work, we provide a new homotopy
theoretic way of detecting nontrivial smooth $U(1)$- and
$Sp(1)$-actions on exotic spheres.

The bridge between exotic spheres and stable homotopy theory is the Pontryagin--Thom isomorphism
$$\cp: \Omega^{\fr}_n \cong \pi_n^s,$$
which identifies the $n$-th stably framed
bordism group of a point with the $n$-th stable homotopy group of the
sphere. As every smooth homotopy sphere is stably framed, we may study exotic spheres by studying their associated elements in the stable homotopy groups of spheres. 

In this work, we study real, complex, and quaternionic Mahowald
invariants
\begin{equation*}
M_{\r}, M_\c, M_{\h} : \pi_*^s \rightsquigarrow \pi_*^s,
\end{equation*}
which carry elements in the stable homotopy groups of spheres to cosets in the stable
homotopy groups of spheres, and their corresponding constructions at the level of framed bordism groups
\begin{equation*}
M_{\r}, M_\c, M_{\h} : \Omega^{\fr}_* \rightsquigarrow \Omega^{\fr}_*.
\end{equation*}
If $\Sigma$ and $\Sigma'$ are framed exotic spheres, Stolz proved
under certain hypotheses that if the relation $\Sigma' \in
M_{\r}(\Sigma)$ holds, then $\Sigma'$ admits a smooth $C_2$-action
with fixed points $\Sigma$. Our main result removes some hypotheses
from Stolz's result concerning $C_2$-actions and
extends his geometric ideas to the case of
higher-dimensional symmetries: we prove that if the relation $\Sigma'
\in M_{\c}(\Sigma)$ holds in $\Omega^{\fr}_*$, then $\Sigma'$ admits a
smooth $U(1)$-action with fixed points $\Sigma$, and if the relation
$\Sigma' \in M_{\h}(\Sigma)$ holds, then $\Sigma'$ admits a smooth
$Sp(1)$-action with fixed points
$\Sigma$. Furthermore, iteration of this construction
  yields infinite families of exotic spheres equipped with $U(1)$-
  and $Sp(1)$-actions. 

The rest of the introduction is divided into two parts. In \cref{SS:Homotopy}, we will discuss the new results from stable homotopy theory which we prove in order to define quaternionic Mahowald invariants as well as their ramifications for computing the more classically studied real and complex Mahowald invariants. In \cref{SS:Geometry}, we turn to the geometric interpretation of the complex and quaternionic Mahowald invariants described above. We discuss some explicit examples, applications to producing torus actions on exotic spheres, potential connections to the chromatic filtration from stable homotopy theory, and potential applications to studying the curvature of exotic spheres. 

\subsection{Homotopy theory}\label{SS:Homotopy}

The real Mahowald invariant, introduced by Mahowald in \cite{Mah67} and frequently studied in subsequent work (see \cite{Qui22} for a survey), is classical, while the complex version, introduced by Ravenel in \cite{Rav84b}, is quite interesting but has not been studied much. The first main contribution of this work, which is purely homotopical, is to define a quaternionic version of the Mahowald invariant (\cref{Def:MI})  and clarify the relationship between the real, complex, and quaternionic Mahowald invariants. 

To provide some context for our result, we recall that the real Mahowald invariant relies on Lin's Theorem \cite{LDMA80}, which gives an equivalence of spectra after $2$-completion
$$\r P^\infty_{-\infty} := \lim_{n \to \infty} \r P^\infty_{-n} := \lim_{n \to \infty} \Th(-n\gamma \to \r P^\infty) \simeq S^{-1}.$$
Here, $\gamma \to \r P^\infty$ is the tautological line bundle and $S^{-1}$ is a single desuspension of the sphere spectrum whose homotopy groups are the stable homotopy groups of spheres. Ravenel defined the complex Mahowald invariant by showing that $S^{-2}$ is a retract of $\c P^\infty_{-\infty}$, which is defined similarly using the tautological complex line bundle over $\c P^\infty$. 

\begin{thmx}\label{MT:Lin}
The following statements hold:
\begin{enumerate}
\item (\cref{Thm:HP}) The spectrum $S^{-4}$ is a retract of $\h P^\infty_{-\infty}$.
\item (\cref{prop:relations-Mahowald-invariants}) 
Let $\alpha \neq 1 \in \pi_n^s$. Then either
  $M_{\c}(\alpha) = M_{\r}(\alpha)$ in $\coker(J)$ or $M_{\c}(\alpha) =
  \alpha$. Similarly, either $M_{\h}(\alpha) = M_{\c}(\alpha) = M_{\r}(\alpha)$ in $\coker(J)$, or
  $M_{\h}(\alpha) = \alpha$. Here, $\coker(J)$ is the cokernel of the stable $J$-homomorphism $J: \pi_nO \to \pi_n^s$. 
\end{enumerate}
\end{thmx}

Part (1) of the theorem allows us to define quaternionic Mahowald invariants, while part (2) explains the relationship between the different types of Mahowald invariants. Perhaps unsurprisingly, the first part of the theorem is proven using stable homotopy theory, adapting Ravenel's argument from the complex case using the modified Adams spectral sequence. More interestingly, the second part of the theorem, which we do not know how to prove using stable homotopy theory, relies on the geometric interpretation of Mahowald invariants described below. 

On one hand, part (2) allows us to easily deduce complex and quaternionic Mahowald invariants from previously computed real Mahowald invariants, allowing us to extend known smooth $C_2$-actions on exotic spheres to smooth $U(1)$- and $Sp(1)$-actions. On the other hand, the computations of \cref{Sec:Applications} suggest that complex and quaternionic Mahowald invariants can be computed more efficiently than their real counterparts. We suggest two projects in this direction:

\subsubsection{Nontriviality of complex and quaternionic Mahowald invariants} 

The real Mahowald invariant of any $x \neq 1 \in \pi_*^s$ will never contain $x$, but this is not the case for the complex and quaternionic Mahowald invariants. Ravenel notes \cite[Pg. 424]{Rav84b} that $\alpha \in M_{\c}(\alpha)$ if and only if $\alpha$ is not in the image of the map induced by the $U(1)$-transfer $\Sigma \c P^\infty_0 \to S^0$. An analogous result holds with $M_{\h}(\alpha)$ and the $Sp(1)$-transfer. This motivates the following:

\begin{problem}
Determine the images of the maps 
$$\pi_{*-1}(\c P^\infty) \to \pi_*^s \quad \text{ and } \quad  \pi_{*-3}(\h P^\infty) \to \pi_*^s$$
induced by the $U(1)$- and $Sp(1)$-transfers. 
\end{problem}

\subsubsection{Efficient computation of real Mahowald invariants}

Let $\nu \in \pi_3^s$ and $\sigma \in \pi_7^s$ denote the quaternionic and octonionic Hopf invariant one maps. It is classical that $\sigma^3 \in M_{\r}(\nu^3)$, and we show below that in fact $\sigma^3 \in M_{\c}(\nu^3) = M_{\h}(\nu^3)$. In the real case, this follows from studying the homotopy groups of $\r P^\infty_{-13}$, but in the quaternionic case, it follows from studying the homotopy groups of $\h P^\infty_{-4}$. The minimal stable cell structure  for $\r P^\infty_{-13}$ has cells in every dimension above $-13$, while the minimal stable cell structure for $\h P^\infty_{-4}$ has cells only in dimensions $-16$, $-12$, $-8$, $-4$, and so on. The relative sparsity of cells in the quaternionic case makes explicit computations much simpler, so in cases where the quaternionic Mahowald invariant is nontrivial, it should be more efficient to compute the real Mahowald invariant using  quaternionic stunted projective spectra instead of real stunted projective spectra.

\begin{problem}
Compute new real Mahowald invariants using quaternionic Mahowald invariants. 
\end{problem} 

\subsection{Geometry}\label{SS:Geometry}

As mentioned above, our main result is a geometric interpretation of the complex and quaternionic Mahowald invariants and connection to symmetries of exotic spheres. 

Let $[\Sigma^{k+n},\bar{b}]$ and $[\Sigma^k,\bar{a}]$ denote framed homotopy spheres. In \cite{Sto88}, Stolz proved that, under certain hypotheses on $k$
and $n$, if $ [\Sigma^{k+n},\bar{b}] \in M_{\r}([\Sigma^k,\bar{a}]) $,
then there exists a smooth $C_2$-action on $\Sigma^{k+n} \# \Sigma'$,
for some homotopy sphere $\Sigma'$ which bounds a parallelizable
manifold, with fixed points $\Sigma^k$. 

Our main theorem extends Stolz's result to actions by positive-dimensional compact Lie groups while also removing some numerical hypotheses. To state our theorem, we say that if $\alpha \in \pi_k^s$ and $M_{\r}(\alpha) \subseteq
\pi_{k+n}^s$, then the \emph{real Mahowald filtration} of $\alpha$ is
$n$. The complex and quaternionic Mahowald filtrations of $\alpha$ are defined
similarly. 

\begin{thmx}[{\cref{Thm:MI} and \cref{Rmk:RemoveJ}}]\label{MT:MI} The following statements hold:
\begin{enumerate}
\item Let $(\Sigma^k, \Sigma^{k+2n})$ be a pair of homotopy
  spheres. Suppose there exist framings $\bar{a}$ and $\bar{b}$ for
  $\Sigma^k$ and $\Sigma^{k+2n}$, respectively, such that
\begin{enumerate}
\item The codimension $2 n$ is bounded above by the complex Mahowald
  filtration $m$ of $[\Sigma^k,\bar{a}]$, and
\item we have
  $$
  [\Sigma^{k+2n},\bar{b}] \in
  \begin{cases}
M_{\c}([\Sigma^k,\bar{a}]) \quad & \text{ if } 2n=m, \\
0 \quad & \text{ if } 2n<m.
  \end{cases}
$$
\end{enumerate}
Then there exists a smooth $U(1)$-action on $\Sigma^{k+2n} \#
\Sigma'$, for some homotopy $(k+2n)$-sphere $\Sigma'$ which bounds a
parallelizable manifold, with fixed points $\Sigma^k$.

\item Let $(\Sigma^k, \Sigma^{k+4n})$ be a pair of homotopy
  spheres. Suppose there exist framings $\bar{a}$ and $\bar{b}$ for
  $\Sigma^k$ and $\Sigma^{k+4n}$, respectively, such that
\begin{enumerate}
\item The codimension $4n$ is bounded above by the quaternionic
  Mahowald filtration $m$ of $[\Sigma^k,\bar{a}]$,
  and
\item we have
$$[\Sigma^{k+4n},\bar{b}] \in  \begin{cases}
M_{\h}([\Sigma^k,\bar{a}]) \quad & \text{ if } 4n=m, \\
0 \quad & \text{ if } 4n<m.
\end{cases}
$$
\end{enumerate}
Then there exists a smooth $Sp(1)$-action on $\Sigma^{k+4n} \#
\Sigma'$, for some homotopy $(k+4n)$-sphere $\Sigma'$ which bounds a
parallelizable manifold, with fixed points $\Sigma^k$.
\end{enumerate}
\end{thmx}

We apply \cref{MT:MI} to deduce the existence of nontrivial $U(1)$-
and $Sp(1)$-actions on some explicit exotic spheres in
\cref{Sec:Applications}. By iterating the Mahowald invariant, we are
able to produce actions not only of $C_2$, $U(1)$, and $Sp(1)$, but
also products of these groups. We summarize with the following:

\begin{thmx}[{\cref{Cor:Actions}(3)}]\label{MT:App}
\
\begin{enumerate}
\item The homotopy spheres $S^7$ and $\Sigma^{21}$ corresponding to
the elements $\sigma$ and $\sigma^3$ in the stable stems admit
  nontrivial $(U(1) \times Sp(1))$- and $U(1)^{\times 2}$-actions with
  fixed points the homotopy spheres $S^1$ and $S^3$ corresponding to
  $\eta$ and $\eta^3$, respectively.
\item Let $\beta \in M_{\c}^{(d)}(\alpha) =
  M_{\c}(M_{\c}(\cdots(M_{\c}(\alpha))))$ be an element in the
  $d$-fold complex Mahowald invariant of $\alpha$ and suppose
  $\Sigma'$ and $\Sigma$ are homotopy spheres detected by $\beta$ and
  $\alpha$. Then, after potentially summing with a homotopy sphere
  which bounds a parallelizable manifold, there exists a smooth $T^d =
  U(1)^{\times d}$-action on $\Sigma'$ with fixed points $\Sigma$. An
  analogous result holds for $Sp(1)$-torus actions on iterated
  quaternionic Mahowald invariants.
\end{enumerate}
\end{thmx}

These geometric applications suggest several directions for future inquiry:

\subsubsection{Higher symmetries and the chromatic filtration}

In \cite[Conj. 12]{MR84}, Mahowald and Ravenel conjecture that (roughly speaking) the real Mahowald invariant increases chromatic height. This conjecture is supported by explicit computations at low heights (see \cite[Pg. 2]{Qui22b} for a summary). We suggest the following generalization of their redshift conjecture:

\begin{conjecture}
If $Y$ is a type $n$ finite complex with $v_n$-periodic self-map $v$ and $\alpha \in \pi_*(\Sigma^{-d_{\k}} Y)$ is $v$-periodic, then the coset $M_{\k}(\alpha)$ consists of entirely $v_n$-torsion elements provided $\alpha \not\in M_{\k}(\alpha)$. 

Let $w: \Sigma^d Y \to Y$ be a power of $v$ which annihilates every element in $M_{\k}(\alpha)$. Let $Z = \cofib(w)$, so $Z$ has type $n+1$ and every element in $M_{\k}(\alpha)$ extends to a map from $Z$ to a suitable sphere. At least one of these maps is $v_{n+1}$-periodic. 
\end{conjecture}

On the other hand, \cref{MT:MI} implies that a framed homotopy sphere which is a $d$-fold iterated complex Mahowald invariant admits a nontrivial $T^d$-action. This motivates the following question about the relationship between the chromatic filtration of the stable homotopy groups of spheres and the smooth transformation groups of exotic spheres:

\begin{question}
Let $n \geq 1$. Suppose $\alpha \in \pi_k^s$ is $v_n$-periodic and let $\Sigma^k$ be the framed sphere corresponding to $\alpha$ under the Pontryagin--Thom isomorphism. Is the degree of symmetry of $\Sigma^k$ at least $n$? 
\end{question}

\subsection{Extension to other compact Lie groups}

In their application of Mahowald invariants to the Geography Problem for $4$-manifolds, Hopkins--Lin--Shi--Xu defined \cite[Def. 1.25]{HLSX22} the $G$-equivariant Mahowald invariant of $\alpha \in \pi_\star^G(S^0)$ with respect to a non-nilpotent element $\beta \in \pi_{\star}^G(S^0)$ for any compact Lie group $G$. Here, $\pi_\star^G(S^0)$ denotes the $RO(G)$-graded $G$-equivariant stable stems.

If $G = C_2$, a result of Bruner and Greenlees \cite{BG95} implies that one can recover the real Mahowald invariant from the $C_2$-equivariant Mahowald invariant with respect to the Euler class
$$[a_\sigma: S^0 \hookrightarrow S^\sigma] \in \pi_{-\sigma}^{C_2}(S^0).$$
Similar arguments imply that the complex and quaternionic Mahowald invariants can be recovered from the $U(1)$- and $Sp(1)$-equivariant Mahowald invariants with respect to the Euler classes
$$[a_\lambda: S^0 \hookrightarrow \c^+] \in \pi_{-\lambda}^{\t}(S^0),$$
$$[a_{\lambda'}: S^0 \hookrightarrow \h^+] \in \pi_{-\lambda'}^{Sp(1)}(S^0),$$
respectively. 

If $G = SO(n)$, then the Euler class
$$[a_{\lambda_n}: S^0 \hookrightarrow (\r^n)^+] \in \pi_{-\lambda_n}^{SO(n)}(S^0)$$
is non-nilpotent, and we can define similar classes for $SU(n)$ and $Sp(n)$ using $\c^n$ and $\h^n$. Following the procedure from \cite[Rmk. 1.26(2)]{HLSX22}, we may define a $G$-equivariant Mahowald invariant
$$M^G: \pi_*^s \rightsquigarrow \pi_*^s$$
carrying elements from the stable stems to cosets in the stable stems.  

\begin{conjecture}
Let $G \in \{SO(n), SU(n), Sp(n)\}$. Suppose that $\beta \in
M^{G}(\alpha)$ with $\beta \neq \alpha$. Then the homotopy sphere
corresponding to $\beta$ admits a smooth $G$-action with fixed points
the homotopy sphere corresponding to $\alpha$.
\end{conjecture}

Balderrama and the second author are studying certain purely homotopical versions of the $SO(n)$-Mahowald invariant in work in progress. 

\subsection{{Curvature bounds}}
Our results also relate to the long-standing problem of understanding
the curvature of exotic spheres. The most well-understood form of
curvature in this setting is scalar curvature.

Let $\alpha : \Omega^{\mathrm{spin}}_n\to KO_n$ be
  the Atiyah-Bott-Shapiro homomorphism which evaluates the index of
  the Dirac operator.  Fundamental results of
Gromov--Lawson \cite{GL80} and Stolz \cite{Sto92}
imply that a simply connected closed spin manifold of dimension $n
\geq 5$ admits a metric with positive scalar curvature if and only if
its $\alpha$-invariant is trivial. Since every homotopy sphere of
dimension $n \geq 3$ is spin, this completely
determines which homotopy spheres admit Riemannian metrics of positive
scalar curvature. Namely, only those homotopy spheres
  $\Sigma^n$ with $0\neq \alpha([\Sigma^n])\in KO_n$, $n = 8k+1$ or
  $n=8k+2$, $k\geq 1$, do not admit a metric of positive scalar
  curvature.

It is well-known that nontrivial smooth group
actions are closely connected to scalar curvature. In \cite{LY74},
Lawson--Yau proved that if a compact manifold admits a smooth,
effective action by any compact, connected, nonabelian Lie group (that
is, if it admits a nontrivial $Sp(1)$-action), then it admits a
Riemannian metric of positive scalar curvature.
Combined with
\cref{MT:MI}, we obtain:
%
%
\begin{corollary}
 {A homotopy sphere $\Sigma^{k+4n}$ representing a
   nontrivial quaternionic Mahowald invariant admits a Riemannian
   metric of positive scalar curvature.}
\end{corollary}
%
%
%
%

While the scalar curvature of exotic spheres is completely understood,
metrics of positive Ricci and nonnegative sectional curvature are
still mysterious. Wraith \cite{Wra97} proved that every exotic sphere
which bounds a stably parallelizable manifold admits a Riemannian
metric of positive Ricci curvature, as does a certain exotic
$8$-sphere which does not bound a stably parallelizable manifold.
The celebrated work of Gromoll--Meyer \cite{GM74}, Grove--Ziller
\cite{GZ00}, and Goette--Kerin--Shankar \cite{GSK20} shows that every
exotic $7$-sphere admits a metric of nonnegative sectional
curvature. As far as we are aware, these are the only known
examples of exotic spheres admitting metrics of positive Ricci and
nonnegative sectional curvature.

There are many results relating group actions on manifolds to the
existence of Riemannian metrics with prescribed curvature bounds,
e.g., the work of Lawson--Yau \cite{LY74} for positive scalar
curvature, Searle--Wilhelm \cite{SW15} for positive Ricci curvature,
and Grove--Ziller \cite{GZ00} for nonnegative sectional curvature.
%
In future work, we hope to address the
  following (perhaps naive) problem:
\begin{problem}
Suppose that the homotopy spheres $\Sigma^k$ and $\Sigma^{k+d_{\k}n}$
are related as in \cref{MT:MI}, and suppose that $\Sigma^k$ admits a
Riemannian metric with a lower bound on its scalar, Ricci, or
sectional curvature. {Find out whether} the homotopy
sphere $\Sigma^{k + d_{\k}n}$ admits {a
  $G_{\k}$-invariant} Riemannian metric with the same curvature
bound.
\end{problem}

\subsection{Linear outline}
In \cref{Sec:HP}, we discuss real, complex, and quaternionic stunted
projective spectra and prove part (1) of \cref{MT:Lin}.

In \cref{Sec:MI}, we define real, complex, and quaternionic Mahowald
invariants in terms of the stable homotopy groups of spheres and
stunted projective spectra. We then express these Mahowald invariants in
terms of equivariant bordism groups by tracing through the Pontryagin--Thom isomorphism.

In \cref{Sec:MT}, we prove \cref{MT:MI}. As a consequence of the proof, we also quickly obtain a proof of part (2) of \cref{MT:Lin}.

In \cref{Sec:Applications}, we compute some low-dimensional complex
and quaternionic Mahowald invariants. We feed these into \cref{MT:MI}
to deduce the existence of interesting $U(1)$- and $Sp(1)$-actions on
some homotopy spheres, proving \cref{MT:App}.

\subsection{Conventions}\label{subsec:conv}
\begin{enumerate}
\item We work integrally in \cref{Sec:HP} and
  in the $p$-complete setting in every section after that.
\item If $X$ is a spectrum, we write {$\widehat{X} = \lim_i X/(p^i)$ for the $p$-completion of $X$} and $\tilde{X} = \colim_i \Sigma^{-1} X/(p^i)$ for
  the $p$-adic cocompletion of $X$.
\item We write $\k$ for $\r$, $\c$, and $\h$ and define $d_{\k} := \dim_{\r} \k$. 
\item We write $G_{\k}$ for the group of units in $\k$.
\item We write `ASS' for the Adams spectral sequence and `MASS' for the
  modified Adams spectral sequence.
\end{enumerate}

\subsection{Acknowledgments}
The authors thank William Balderrama, Dan Dugger, John Greenlees, Doug Ravenel, John
Rognes, XiaoLin Danny Shi, Stephan Stolz, and Zhouli Xu for helpful
discussions. The first author was
  partially supported by Simons collaboration grant 708183. The second
  author was partially supported by NSF grants DMS-2039316 amd
  DMS-2314082, as well as an AMS-Simons Travel Grant.
 The second author also thanks the Max Planck Institute for
providing a wonderful working environment and financial support during
the beginning of this project. 

\section{Stunted real, complex, and quaternionic projective spectra}
  \label{Sec:HP}
In this section, we discuss the stunted projective spectra needed to
define the real, complex, and quaternionic Mahowald invariants. Our
main theorem is \cref{Thm:HP}(3), the quaternionic analogue of Lin's
Theorem \cite{LDMA80}, which allows us to define quaternionic Mahowald
invariants.

\begin{definition}
For each $\k \in \{ \r, \c, \h \}$, let $\gamma_\k$ denote the
tautological $\k$-line bundle over $\k P^\infty$.

For each $i \in \z$, the \emph{$i$-th stunted $\k$-projective
spectrum} $\k P^\infty_i$ is the Thom spectrum of the $i$-fold Whitney
sum of $\gamma_\k$ over $\k P^\infty$,
$$\k P^\infty_i := \Th(\k P^\infty, i \gamma_{\k}).$$
Let
$$\k P^\infty_{-\infty} := \lim_i \k P^\infty_{-i}$$ be the inverse
limit taken along the maps induced by the inclusions $i \gamma_{\k}
\to (i+1) \gamma_{\k}$.
\end{definition}
The main theorem we need about these stunted projective spectra is the
following. 
\begin{theorem}\label{Thm:HP}
For each $\k \in \{\r, \c, \h\}$, some desuspension of the {$2$-complete} sphere
spectrum is a retract of $\k P^\infty_{-\infty}$. More precisely:
\begin{enumerate}
\item (Lin, \cite{LDMA80}) There is an equivalence of spectra $S^{-1}
  \simeq \r P^\infty_{-\infty}$.
\item (Ravenel, \cite{Rav84b}) The spectrum $S^{-2}$ is a retract of $
  \c P^\infty_{-\infty}$.
\item The spectrum $S^{-4}$ is a retract of $\h P^\infty_{-\infty}$. 
\end{enumerate}
Here, we emphasize that the statements hold only after $2$-completion.
\end{theorem}

We will prove Part (3) of \cref{Thm:HP} by modifying Ravenel's proof
of Part (2) from \cite{Rav84b}.
\begin{remark}
We focus on the $2$-primary setting in this paper, but point out the
following:
\begin{enumerate}
\item The odd-primary analogue of Lin's Theorem is due to Gunawardena
  \cite{AGM85}; one replaces $\r P^\infty_{-\infty}$ by an appropriate
  inverse limit of stunted lens spectra.
\item Ravenel's result concerning $\c P^\infty_{-\infty}$ and our
  result concerning $\h P^\infty_{-\infty}$ hold after completion at
  \emph{any} prime. We only make $2$-primary computations in this
  paper, but odd-primary computations could also be interesting.
\end{enumerate}
\end{remark}
\subsection{Proof of \cref{Thm:HP}}
Our proof of \cref{Thm:HP}(3) is a fairly straightforward modification
of Ravenel's proof of \cref{Thm:HP}(2). To avoid rewriting Ravenel's
entire paper with `$\h$' in place of `$\c$', we will state the
quaternionic analogue of each key result from \cite[Secs. 1 and
  2]{Rav84b}, adding details and proofs only when they differ
substantially from the complex case (or when the proof of the complex
case is omitted by Ravenel).
\begin{convention}\label{conv-skel}
In the remainder of this section only, we will write $\h P_i := \h
P^\infty_i$. We write $X^{(j)}$ for the $j$-skeleton of $X$.
\end{convention}
\begin{lemma}[{Analogue of \cite[Lem. 1.4]{Rav84b}}]\label{Lem:Rav1.4}
  There is an equivalence
\begin{equation*}
    \h P_i / \h P^{(4i)}_i \simeq \h P_{i+1}.
\end{equation*}    
\end{lemma}
The following lemma can be proven using James periodicity for
quaternionic stunted projective spectra \cite{Ati61}.
\begin{lemma}[{Analogue of \cite[Lem. 1.6]{Rav84b}}]\label{Lem:Rav1.6}
For $j>i$, there is an equivalence $D \h P^{(4j-4)}_i \simeq \Sigma^4
\h P^{-4i-4}_{-j}$.
\end{lemma}
As mentioned above, the bundle map $-n \gamma_{\h} \to (-n+1)
\gamma_{\h}$ gives rise to an inverse system
$$\h P_0 \leftarrow \h P_{-1} \leftarrow \h P_{-2} \leftarrow
\cdots.$$ Note also that since $\h P_0 \simeq \Sigma^\infty \h
P^\infty_+ \simeq \Sigma^\infty BSp(1)_+$, we may define a $p$-th
power map
\begin{equation*}
  [p]: \h P_0 \to \h P_0.
\end{equation*} 
We defer the proof of the following result to \cref{SS:MASS}.

  Recall the conventions of \cref{subsec:conv} and
  \cref{conv-skel}.
\begin{thm}[{Analogue of \cite[Thm. 1.7]{Rav84b}}]\label{Thm:CohoIso}
For $i>0$, the composite
$$\widetilde{\h P_{-i}} \to \widetilde{\h P_0} \xrightarrow{[p]} \widetilde{\h P_0}$$
induces an isomorphism in the cohomotopy group $\pi^k$ for $k > 1-2i$. 
\end{thm}

For $i>0$, we have cofiber sequences
\begin{equation}\label{Eqn:gDefSeq}
\h P_{-k-i} \to \h P_{-k} \to \Sigma \h P^{(-4k-4)}_{-k-i}
\end{equation}
by \cref{Lem:Rav1.4}, and using \cref{Lem:Rav1.6}, the right-hand map dualizes to
$$\Sigma^3 \h P^{(4k+4i-4)}_k \to D \h P_{-k}.$$
Letting $i$ go to $\infty$, we get a map
\begin{equation}\label{Eqn:gMap}
g: \Sigma^3 \h P_k \to D \h P_{-k}
\end{equation}
for all integers $k$. 

For $k\geq0$, we also have the map
\begin{equation}\label{Eqn:Dp}
D[p] : D\h P_0 \xrightarrow{D[p]} D\h P_0 \to D \h P_{-k},
\end{equation}
where the last map is the dual of the map induced by the map $-k\gamma_{\h} \to 0$. 

For each $k \geq 0$, we obtain\footnote{There is a typo in \cite{Rav84b} here: $k \leq 0$ should be $k \geq 0$.}
\begin{equation}\label{Eqn:eMap}
e: \Sigma^3 \h P_k \vee \bigvee_{i=1}^\infty \Sigma^3 \h P_0 \to D\h P_{-k}.
\end{equation}
The map on the left-hand summand is $g$ and the map on the $i$-th summand is the composite
$$\Sigma^3 \h P_0 \xrightarrow{[p]^i} \Sigma^3 \h P_0 \to \Sigma^3 \h P_k \xrightarrow{g} D\h P_{-k}.$$
The assumption $k \geq 0$ is required since the unlabeled map is induced by the inclusion $0 \to k\gamma_{\h}$. 

\begin{theorem}
For $k \geq 0$, there is a map
$$e': D\widetilde{\h P_{-k}} \to \Sigma^3 \h P_k \vee \prod_{i=1}^\infty \Sigma \widehat{\h P}_0$$
such that
\begin{enumerate}
\item $e'e$ is the composite of $p$-adic completion and the inclusion of a sum into a product,
\item $e'$ has fiber $\widehat{S}^0$, and
\item $e'$ is a retraction, so $D \h P_{-k} \simeq \widehat{S^0} \vee \Sigma^3 \widehat{\h P}_k \vee \prod_{i=1}^\infty \Sigma \widehat{\h P}_0$. 
\end{enumerate}
\end{theorem}

\begin{proof}
Using that $\pi^k(X) \cong \pi_{-k}(DX)$ and dualizing \cref{Thm:CohoIso}, we see that the composite
$$D\widetilde{\h P_{0}} \xrightarrow{D[p]} D\widetilde{\h P_{0}} \to D\widetilde{\h P_{-i}}$$
is a $(4i-2)$-equivalence. Letting $i \to \infty$ shows that
$$D \widetilde{\h P_0} \simeq \lim_{i \to \infty} D \widetilde{\h P_{-i}}$$
is an equivalence.

Dualizing \eqref{Eqn:gDefSeq} and letting $i \to \infty$, we obtain a cofiber sequence
$$\Sigma^3 \widehat{\h P}_k \xrightarrow{g} D \widetilde{\h P_{-k}}
\to \lim_{i \to \infty} D \widetilde{\h P_{-i}}$$ for $k \geq 0$. The
right-hand spectrum is equivalent to $D \widetilde{\h P_0}$, so the
right-hand map is a retraction onto $D \widetilde{\h P_0}$ split by
$D[p]$, and the left-hand map is a splitting for
$g$.
Thus we have
$$D \widetilde{\h P_{-k}} \simeq \Sigma^3 \widehat{\h P}_k \vee D \widetilde{\h P_0}$$
for $k \geq 0$. 

If $k=0$, then we have a cofiber sequence
$$D \h P_0 \xrightarrow{D[p]} D \h P_0 \to \Sigma \widehat{\h P_0}.$$
By induction on $n$, we obtain a diagram
\[
\begin{tikzcd}
D \widetilde{\h P_0} \arrow{r}{D[p^n]} & D \widetilde{\h P_0} \arrow{r} & \bigvee_{i=1}^n \Sigma \widehat{\h P_0} \\
D \widetilde{\h P_0} \arrow{u}{D[p]} \arrow{r}{D[p^{n+1}]} & D \widetilde{\h P_0} \arrow{u}{=} \arrow{r} & \bigvee_{i=1}^{n+1} \Sigma \widehat{\h P_0}
\end{tikzcd}
\]
where the rows are cofiber sequences and the right-hand vertical map collapses the $(n+1)$st wedge summand. Letting $n \to \infty$, we get a cofiber sequence
$$\lim_{D[p]} D \widetilde{\h P_0} \to D \widetilde{\h P_0} \xrightarrow{e''} \prod_{i=1} \Sigma \widehat{\h P_0}.$$
The map $e'$ from the theorem statement is the composite
$$e' : D \widetilde{\h P_{-k}} \xrightarrow{\simeq} \Sigma^3 \widehat{\h P_k} \vee D \widetilde{\h P_0} \xrightarrow{ id \vee e''} \Sigma^3 \widehat{\h P_k} \vee \prod_{i=1} \Sigma \widehat{\h P_0}.$$

Now, part (1) is clear by construction. For part (2), we have
$$\lim_{D[p]} D \widetilde{\h P_0} \simeq D \lim_{[p]} \widetilde{\h P_0}.$$
But 
$$\lim_{[p]} \widetilde{\h P_1} \simeq \lim_{\cdot p} \widetilde{\Sigma BSp(1)} \simeq *,$$
so
$$\lim_{[p]} \widetilde{\h P_0} \simeq \widetilde{S^0}$$
and $e'$ has fiber $\widehat{S^0}$. Since this $\widehat{S^0}$ is dual to the bottom cell of $\h P_0$, which is a retract, $e'$ must be a retraction; this proves part (3). 
\end{proof}

Turning the triangle \eqref{Eqn:gDefSeq}, we have a cofiber sequence
$$\h P^{(4k-4)}_{k-i} \to \h P_{k-i} \to \h P_k$$
which dualizes to give
$$D \h P_k \to D \h P_{k-i} \to \Sigma^4 \h P^{(4i-4k-4)}_{-k}.$$
Completing, letting $i \to \infty$, and identifying middle terms as in the previous proof, we obtain a cofiber sequence
$$D \widetilde{\h P_k} \to D \widetilde{\h P_0} \to \Sigma^4 \widehat{\h P_{-k}}.$$
Taking the inverse limit as $k \to \infty$, the fiber becomes contractible, proving:

\begin{corollary}
There is an equivalence $\lim_k \widehat{\h P_{-k}} \simeq \Sigma^{-4} D \widetilde{\h P_0}$. 
\end{corollary}

\subsection{Proof of \cref{Thm:CohoIso}}\label{SS:MASS}

Recall the \emph{modified Adams spectral sequence} (\cite[Sec. 2]{Rav84b}). If 
$$Y = Y_0 \leftarrow Y_1 \leftarrow Y_2 \leftarrow \cdots$$
is an Adams resolution of $Y$ and
$$X = X_0 \to X_1 \to X_2 \to \cdots$$
is a diagram with $H^*(\lim_i X_i) = 0$, then there is a homotopy commutative diagram
\[
\begin{tikzcd}
F(X_0,Y_0) & \arrow{l} F(X_0,Y_1) & \arrow{l} \cdots \\
F(X_1,Y_0) \arrow{u} & F(X_1,Y_1) \arrow{u} \arrow{l} & \arrow{l} \cdots \\
\vdots \arrow{u} & \vdots \arrow{u} & . 
\end{tikzcd}
\]
Taking $W_s = \bigcup_{i+j = s} F(X_i,Y_j)$, the MASS for $[X,Y]_*$ is the spectral sequence associated with the homotopy exact couple given by the diagram
$$F(X,Y) = W_0 \leftarrow W_1 \leftarrow W_2 \leftarrow \cdots.$$
The key properties of the MASS are summarized in the following:

\begin{lemma}[{\cite[Lems. 2.12 and 2.16]{Rav84b}}]\label{Lem:MASS}
\
\begin{enumerate}
\item With notation as above, suppose that each map $X_i \to X_{i+1}$ is injective in mod $p$ cohomology. Then the $E_2$-term for the MASS for $[X,Y]_*$ is given by
$$E_2^{s,t} = \bigoplus_{i \geq 0} \Ext_A^{s-i,t+1-i}(H^*(Y), H^*(X_{i+1},X_i)).$$
\item Suppose each map $X_i \to X_{i+1}$ is trivial in mod $p$ cohomology. Then the resulting MASS for $[X,Y]_*$ is isomorphic (from $E_2$ onward) to the standard ASS for $[X,Y]_*$. 
\end{enumerate}
\end{lemma}

The $E_2$-term of the (M)ASS for $\widetilde{\c P^\infty_{-i}}$ is computed in \cite[Lem. 2.5]{Rav84b}, while the $E_2$-term for the MASS for the cohomotopy of $\widetilde{\c P^\infty_0}$ can be computed using \cref{Lem:MASS}. The $E_2$-terms coincide, and fairly standard arguments (cf. \cite[pp. 433-434]{Rav84b}) imply that the isomorphism between $E_2$-terms is induced by the map between diagrams. 

To prove \cref{Thm:CohoIso}, we take $Y=S^0$ (so $Y_\bullet$ is an Adams resolution of the sphere) and take $X \to X_\bullet$ to be the diagram
$$\widetilde{\h P_{-i}} = \widetilde{\Th}(\h P^\infty, -i\gamma_{\h}) \to \widetilde{\Th}(\h P^\infty, (-i-1) \gamma_{\h} + \gamma_{\h}^p) \to \widetilde{\Th}( \h P^\infty, (-i-2) \gamma_{\h} + 2 \gamma^p) \to \cdots.$$
Each map induces zero in mod $p$ cohomology (cf. \cite[Lem. 2.17]{Rav84b}), so the associated MASS coincides with the standard ASS for the cohomotopy of $\h P_{-i}$ from $E_2$ onward by \cref{Lem:MASS}. Moreover, this diagram maps to the diagram
$$\widetilde{\h P_0} \to \widetilde{\h P_1} \to \widetilde{\h P_2} \to \cdots$$
which gives the MASS for the cohomotopy of $\widetilde{\h P_0}$. 

This $E_2$-page can be computed using the short exact sequence of $A$-modules
$$0 \to H \to C \to \Sigma^{-1} H \to 0,$$
where $H = \lim_i H^*(\widetilde{\h P_{-i}})$ and $C = \lim_i H^*(\widetilde{\c P_{-i}})$, following the proof of \cite[Lem. 2.5]{Rav84b}. On the other hand, the $E_2$-term of the MASS associated to the bottom row (again with $Y = S^0$) can be computed using \cref{Lem:MASS}. The two $E_2$-terms agree, and Ravenel's argument that the map of diagrams comes from the claimed map follows through with little modification.

\section{Mahowald invariants and equivariant bordism}\label{Sec:MI}
In this section, we relate the real, complex, and quaternionic
Mahowald invariants to $C_2$-, $U(1)$-, and $Sp(1)$-equivariant
bordism groups, respectively. Our main result is
\cref{Prop:Stolz1.18}.

\subsection{Classical definition} 

For $\k \in \{\r, \c, \h\}$, we define $d_{\k} := \dim_{\r} \k$. For
all $n \geq 0$, \cref{Thm:HP} implies that there are compatible maps
\begin{equation}\label{Eqn:prK}
\pr_{\k}: S^{-d_{\k}} \to \k P^\infty_{-n}.
\end{equation}
We use these to make the following definition:
\begin{definition}\label{Def:MI}
Let $\alpha \in \pi_t S^0$. The \emph{$\k$-Mahowald invariant of
$\alpha$} is the coset of completions of the diagram
\[
\begin{tikzcd}
S^{t-d_{\k}} \arrow{d}{\alpha} \arrow[r, dashed,"M_{\k}(\alpha)"] & S^{-d_{\k} \cdot N} \arrow{dd} \\
S^{-d_{\k}} \arrow{d}{\pr_{\k}} \\
\k P^\infty_{-\infty} \arrow{r} & \k P^\infty_{-N} 
\end{tikzcd}
\]
where $N>0$ is minimal such that the left-hand composite is nontrivial. 
\end{definition}

We compute several examples in \cref{Sec:Applications}. In those computations, we observe that $M_{\c}(\alpha) = M_{\r}(\alpha)$ if $\alpha \neq M_{\c}(\alpha)$, and similarly for
$M_{\h}(\alpha)$. The inclusions $\r \hookrightarrow \c
\hookrightarrow \h$ induce inclusions $\r P^\infty \to \c P^\infty \to
\h P^\infty$ along with bundle maps between their tautological line
bundles, which allow us to compare $M_{\r}$, $M_{\c}$, and
$M_{\h}$. 

In fact, we can prove the following relationship between these different kinds of Mahowald invariants. We defer the proof to \cref{Sec:MT}:

\begin{thm}\label{prop:relations-Mahowald-invariants}
Let $\alpha \neq 1 \in \pi_n^s$. Then either
  $M_{\c}(\alpha) = M_{\r}(\alpha)$ in $\coker(J)$ or $M_{\c}(\alpha) =
  \alpha$. Similarly, either $M_{\h}(\alpha) = M_{\r}(\alpha)$ in $\coker(J)$, or
  $M_{\h}(\alpha) = \alpha$.
\end{thm}

\begin{remark}\label{Rmk:MIpis}
Rewriting \cref{Def:MI} in terms of homotopy groups, the $\k$-Mahowald invariant of $\alpha$ is the coset of $\beta$ in $\pi_{t+d_\k N}^s$ such that $j(\beta) = \pr_{\k}(\alpha)$ in the diagram
\[
\begin{tikzcd}
	& \pi_t^s \arrow{d}{\pr_{\k}} \\
\pi_{t+ d_\k N}^s \arrow{r}{j} & \pi_{t-d_\k} \k P^\infty_{-N}.
\end{tikzcd}
\]
\end{remark}

Our goal for the rest of the section is to use the Pontryagin--Thom construction to reinterpret the equality $j(\beta) = \pr_{\k}(\alpha) \in \pi_* \k P^\infty_{-N}$ in terms of equivariant bordism groups. 

\subsection{Definition in terms of normal and equivariant bordism}

We now rewrite the diagram of \cref{Rmk:MIpis} in terms of (nonequivariant) bordism groups. The Pontryagin--Thom isomorphism $\pi_k^s \cong \Omega_k^{\fr}$ identifies the $k$-th stable stem with the $k$-th framed bordism group. More generally, the Pontryagin--Thom construction yields an isomorphism
\begin{equation}\label{Eqn:PT}
\pi_i \Th(X,V) \cong \Omega_i (X,V)
\end{equation}
between the $i$-th homotopy group of the Thom spectrum $\Th(X,V)$ and the \emph{normal bordism group} $\Omega_i(X,V)$:

\begin{definition}
Let $X$ be a topological space and $V$ a virtual vector bundle over $X$. The \emph{$i$-th normal bordism group} $\Omega_i(X,V)$ is the group of triples $(M,f,\bar{f})$ with
\begin{enumerate}
\item $M$ is an $i$-dimensional closed manifold,
\item $f : M \to X$ is a map, and
\item $\bar{f}: f^*V^+ \oplus TM \to f^*V^-$ is a stable isomorphism of vector bundles,
\end{enumerate}
where $V^+$ and $V^-$ are vector bundles over $X$ with $V = V^+ - V^- \in KO^0(X)$. 
\end{definition}

To bring group actions into the picture, we will identify the homotopy groups of stunted projective spectra with certain equivariant bordism groups. We need one more definition:

\begin{definition}
We define the following \emph{equivariant bordism groups}:
\begin{enumerate}
\item Let $\r^{n,k}_\r = \r^{n\sigma + k}$ denote $\r^{n+k}$ where $C_2$ acts by sign on the first $n$ components and trivially on the last $k$ components. Let $\epsilon^{n,k}_\r$ denote the trivial bundle with fibers $\r^{n,k}_\r$. If $M$ is a $C_2$-manifold, a $C_2$-equivariant bundle isomorphism $\bar{g}: TM \oplus \epsilon^{s,t}_\r \xrightarrow{\cong} \epsilon^{n+s,k+t}_\r$ is called an \emph{$(n,k)$-$\r$-framing}. We define $\Omega_{n,k}^\r$ to be the bordism group of free $C_2$-manifolds with an $(n,k)$-$\r$-framing. 

\item Let $\r^{2n,k}_{\c} = \c^n \oplus \r^k$ be the $U(1)$-representation where $U(1)$ acts by rotation on each copy of $\c$ and trivially on each copy of $\r$. Let $\epsilon^{2n,k}_{\c}$ denote the trivial (real) bundle with fibers $\r^{2n,k}_\c$. If $M$ is an $U(1)$-manifold, a $U(1)$-equivariant bundle isomorphism $\bar{g}: TM \oplus \epsilon^{s,t}_{\c} \xrightarrow{\cong} \epsilon^{2n+s,k+t}_{\c}$ is called a \emph{$(2n,k)$-$\c$-framing}. We define $\Omega_{2n,k}^{\c}$ to be the bordism group of free $U(1)$-manifolds with a $(2n,k)$-$\c$-framing. 

\item Let $\r^{4n,k}_{\h} = \h^n \oplus \r^k$ be the $Sp(1)$-representation where $Sp(1)$ acts on each copy of $\h$ in the standard way and trivially on each copy of $\r$. Let $\epsilon^{4n,k}_{\h}$ denote the trivial (real) bundle with fibers $\r^{4n,k}_{\h}$. If $M$ is an $Sp(1)$-manifold, an $Sp(1)$-equivariant bundle isomorphism $\bar{g}: TM \oplus \epsilon^{s,t}_{\h} \xrightarrow{\cong} \epsilon^{4n+s,k+t}_{\h}$ is called a \emph{$(4n,k)$-$\h$-framing}. We define $\Omega_{4n,k}^{\h}$ to be the bordism group of free $Sp(1)$-manifolds with a $(4n,k)$-$\h$-framing. 
\end{enumerate}

More succinctly: let $G_\r = C_2$, $G_\c = U(1)$, and $G_\h = Sp(1)$. Then $\r_\k^{d_\k n,k} = \k^n \oplus \r^k$, where $\k$ is the standard representation of $G_\k$ and $\r$ is the trivial representation, and we get that $\Omega_{d_\k n, k}^{\k}$ is the bordism group of free $G_\k$-manifolds with a $(d_\k n, k)$-$\k$-framing. 
\end{definition}

By \eqref{Eqn:PT}, we have an isomorphism
$$\pi_k(\k P^\infty_{-n}) \cong \pi_k(\Sigma^{-d_{\k}n} \Th(\k P^\infty, -n\gamma_\k)) \cong \Omega_{k+d_{\k} n}(\k P^\infty, -n\gamma_{\k}).$$
The normal bordism groups on the right-hand side admit the following description in terms of equivariant bordism groups:

\begin{proposition}\label{Prop:ProjBord}
There is an isomorphism
$$\Omega_{k+d_{\k}n}(\k P^\infty, -n \gamma_{\k}) \xrightarrow{\cong} \Omega^{\k}_{d_{\k} n, k}.$$
\end{proposition}

\begin{proof}
We follow the proof for the case $\k = \r$ which appears around \cite[Eqn. 1.10]{Sto88}. Let $[N,f,\bar{f}] \in \Omega_{k+d_{\k}n}(\k P^\infty, -n\gamma_{\k})$, so
\begin{enumerate}
\item $N$ is a closed $k+d_{\k}n$-manifold,
\item $f: N \to \k P^\infty$ is a map, and 
\item $\bar{f}: TN \to f^*(n\gamma_{\k})$ is a stable isomorphism.
\end{enumerate}
We define $[\tilde{N}, \tilde{\bar{f}}] \in \Omega^{\k}_{d_{\k} n, k}$ as follows. Let
$$\tilde{N} := S(f^*(\gamma_{\k}))$$
denote the unit sphere bundle in the pullback of the tautological $\k$-line bundle along $f$. Inspecting the $G_\k$ action on $\gamma_{\k}$, we see that $\tilde{N}$ is a free $G_\k$-manifold. We define
$$\tilde{\bar{f}}: T\tilde{N} \oplus \epsilon^{0,t} \to \epsilon^{d_{\k}n, k+t}$$
to be the $G_\k$-equivariant bundle isomorphism induced by the isomorphism
$$\bar{f}: TN \oplus \epsilon^t \to f^*(n\gamma_\k) \oplus \epsilon^{k+t} = (\tilde{N} \times \r_\k^{d_{\k}n, k+t}) / G_\k.$$

The assignment $[N,f,\bar{f}] \mapsto [\tilde{N}, \tilde{\bar{f}}]$ defines an isomorphism
$$\Omega_{k+d_{\k}n}(\k P^\infty,-n \gamma_{\k}) \to \Omega_{d_\k n, k}^\k.$$
Indeed, passage to the quotient manifold gives an inverse. 
\end{proof}

\begin{remark}\label{Rmk:ConnectedSimplification}
Note that if $[N,f,\bar{f}] \in \Omega_{k+d_{\k}n}(\k P^\infty, -n\gamma_\k)$ is an element with $N$  $d_{\k}$-connected, then 
$$\tilde{N} = S(f^*(\gamma_{\k})) \cong G_\k \times N.$$
Explicitly:
\begin{enumerate}
\item If $\k = \r$, then $\tilde{N} = C_2 \times N$ when $N$ is simply connected;
\item If $\k = \c$, then $\tilde{N} = S^1 \times N$ when $N$ is $2$-connected;
\item If $\k = \h$, then $\tilde{N} = S^3 \times N$ when $N$ is $4$-connected.
\end{enumerate}
\end{remark}

We are now ready to describe the $\k$-Mahowald invariant in terms of equivariant bordism. Let
$$\bar{c}_n^\k: TS(\r^{d_\k n,0}_{\k}) \oplus \epsilon_{\k}^{0,1} \xrightarrow{\cong} \epsilon^{d_{\k} n,0}_\k$$
denote the $G_\k$-equivariant vector bundle isomorphism obtained by restricting the isomorphism
$$TD^{d_\k n, 0} \cong \epsilon^{d_{\k} n, 0}_\k$$
to $S(\r^{d_{\k} n,0}_{\k}) \subset D^{d_{\k} n, 0}_\k$ and composing with 
$$TS(\r^{d_{\k}n,0}_{\k}) \oplus \epsilon^{0,1}_{\k} \cong TD^{d_{\k}n,0}_{\k} \mid_{S(\r^{d_{\k}n,0}_{\k})}.$$

For all $m \in \z$, let
$$\cdots \to \pi_*(S^{d_{\k}m}) \xrightarrow{j_m} \pi_*(\k P^\infty_m) \xrightarrow{p_m} \pi_*(\k P^\infty_{m+1}) \xrightarrow{t_{m+1}} \pi_{*-1}(S^{d_{\k}m}) \to \cdots$$
be the long exact sequence in homotopy obtained from the cofiber sequence
$$S^{d_{\k}m} \to \k P^\infty_{m} \to \k P^\infty_{m+1}.$$

\begin{proposition}
Let $(M^k,\bar{a})$ and $(N^{k+d_{\k}n},\bar{b})$ be framed manifolds with $k \geq d_{\k}$. The following statements are equivalent:
\begin{enumerate}
\item $M_{\k}([M,\bar{a}]) = [N,\bar{b}] \in \pi_{k+d_{\k}n}^s/\im(t_{-n})$;
\item $0 \neq [S(\r^{d_{\k}(n+1),0}_{\k}) \times M, \bar{c}_{n+1}^{\k} \times \bar{a}] = [G_{\k} \times N, G_{\k} \times \bar{b}] \in \Omega^{\k}_{d_{\k}(n+1),k-d_{\k}}$. 
\end{enumerate}
\end{proposition}

\begin{proof}
By definition,
$$M_{\k}([M,\bar{a}]) = [N,\bar{b}] \in \pi_{k+d_{\k}n}^s/\im(t_{-n})$$
if and only if 
$$\pr_m([M,\bar{a}]) = 0, \quad \pr_{m-1}([M,\bar{a}]) \neq 0, \quad \text{ and } j_{m-1}([N,\bar{b}]) = \pr_{m-1}([M,\bar{a}]).$$
We must express each of these conditions in terms of equivariant bordism. 

We begin with the map $j$. Using the Pontryagin--Thom construction, we may identify this with
$$j: \Omega_k(S(\gamma_{\k}), p^*(n\gamma_{\k})) \to \Omega_k(\k P^\infty, n\gamma_{\k}),$$
$$[M,f,\bar{f}] \mapsto [M,p \circ f, \bar{f}],$$
and composing with the isomorphism from \cref{Prop:ProjBord}, we obtain
$$j_{-n} : \Omega_{k+d_{\k}n}^{\fr} \to \Omega_{d_{\k}n+1,k-1}^{\k},$$
$$[M,\bar{a}] \mapsto [G_{\k} \times M, G_{\k} \times \bar{a}].$$
Here, we use the assumption that $k \geq d_{\k}$ (cf. \cref{Rmk:ConnectedSimplification}). 

To identify $\pr_m$ in terms of equivariant bordism, we first interpret $p$ geometrically. Using the Pontryagin--Thom construction, we see that
\begin{equation}\label{Eqn:pGeom}
p: \Omega_k(\k P^\infty, n\gamma_{\k}) \to \Omega_{k-d_\k n}(\k P^\infty, (n+1)\gamma_{\k}), 
\end{equation}
$$[M,f,\bar{f}] \mapsto [N=s^{-1}(0), f|_{N}, \bar{f}|_{N}],$$
where $s$ is a section of $f^*\gamma_{\k}$ which is transverse to the zero section. Identifying the source and target of $p$ with equivariant bordism groups via \cref{Prop:ProjBord}, we see that 
$$p: [S(\r_{\k}^{d_{\k}(n+1),0}), \bar{c}_{n+1}^{\k}] \mapsto [S(\r_{\k}^{d_{\k}n,0}), \bar{c}_n^{\k}]$$
as on \cite[Pg. 114]{Sto88}, and thus those classes assemble into a generator
$$g \in \lim_m \Omega^{\k}_{m d_{\k}, {-1}} \cong \lim_m \pi_{-d_{\k}}(\k P^\infty_{-m}).$$
Thus
$$\pr_{-n}: \pi_k^s \xrightarrow{g_*} \lim_{m} \pi_{k-1}(\k P^\infty_{-m}) \to \pi_{-d_{\k}}(\k P^\infty_{-n})$$
may be identified with
$$\pr_{-n}: \Omega_k^{\fr} \to \Omega_{d_{\k}n, k-1}^{\k},$$
$$[M,\bar{a}] \mapsto [S(\r^{d_{k}n,0}_{\k}) \times M, \bar{c}_n^{\k} \times \bar{a}].$$
Putting together these identifications completes the proof. 
\end{proof}

In order to prove our main theorem, we need the following auxiliary construction. 

\begin{definition}[Twisting framings]
Let $M$ be a free $G_{\k}$-manifold with a $(d_{\k} n,k)$-$\k$-framing $\bar{f}: TM \oplus \epsilon_{\k}^{d_{\k}s,t} \to \epsilon_{\k}^{d_{\k}n + d_{\k}s, k+t}$ and let $g \in KO^{-1}_{G_{\k}}(M)$. Regarding $g$ as a $G_{\k}$-equivariant bundle automorphism of $\epsilon^{d_{\k}n+d_{\k}s,k+t}_{\k}$, we may compose to obtain a new $(d_{\k}n,k)$-$\k$-framing of $M$,
$$g\bar{f} : TM \oplus \epsilon^{d_{\k}s,t}_{\k} \xrightarrow{\bar{f}} \epsilon_{\k}^{d_{\k}n+d_{\k}s,k+t} \xrightarrow{g} \epsilon^{d_{\k}n + d_{\k}s, t+k}_{\k}.$$
\end{definition}

\begin{construction}[{The framing $\bar{s}(\bar{t},\bar{a})$}]
Let $M^k$ be a framed $k$-manifold and let $\zeta^{d_{\k}n } \to M^k$ be a stably trivial vector bundle of rank $d_{\k}n$. Let $\bar{t}: \zeta^{d_{\k}n} \oplus \epsilon^s \xrightarrow{\cong} \epsilon^{d_{\k}n+s}$ be a stable trivialization of $\zeta^{d_{\k}n}$ and let $\bar{a} : TM^k \oplus \epsilon^t \xrightarrow{\cong} \epsilon^{t+k}$ be a stable framing of $M^k$. Note that $G_{\k}$ acts on $S(\zeta^{d_{\k}n}) \cong S(\k^n)$ by restriction of the diagonal action of $G_{\k}$ on $\k^n$. 

We obtain a $(d_{\k}n,k-1)$-$\k$-framing on $S(\zeta^{d_{\k}n})$ as follows. If $p: D(\zeta^{d_{\k}n}) \to M^k$, then there is a $G_{\k}$-equivariant bundle isomorphism
$$TD(\zeta^{d_{\k}n}) \cong p^*(\zeta^{d_{\k}n}) \oplus p^*TM^k,$$
where $G_{\k}$ acts on $TM^k$ trivially. The composite
$$TD(\zeta^{d_{\k}n}) \oplus \epsilon^{d_{\k}n,t}_{\k} \cong p^*\zeta^{d_{\k}n} \oplus \epsilon^{d_{\k}s,0}_{\k} \oplus p^*TM \oplus \epsilon^{0,t}_{\k} \xrightarrow{p^*\bar{t} \oplus p^*\bar{a}} \epsilon^{d_{\k}n + d_{\k}s,k+t}_{\k}$$
gives a $(d_{\k}n,k)$-$\k$-framing of $D(\zeta^{d_{\k}n})$, and this restricts to a framing
$$\bar{s}(\bar{t},\bar{a}): TS(\zeta^{d_{\k}n}) \oplus
\epsilon^{d_{\k}s, t-1}_{\k} \xrightarrow{\cong}
\epsilon^{d_{\k}n+d_{\k}s, k+t-1}_{\k}$$ generalizing the framing
$\bar{c}_n^{\k} \times \bar{a}$.
\end{construction}

\begin{construction}[{Twisting by elements in $KO^{-1}(\k P^\infty \wedge S^k)$}]
  Let
\begin{equation*}
    h \in KO^{-1}(\k P^\infty \wedge S^k)
\end{equation*}
    and let $c: M^k \to S^k$ be a degree one map. Using the composite
\begin{equation*}
KO^{-1}\!(\k P^\infty \!\wedge S^k\!)\! \xrightarrow{\pr^*}\!\! KO^{-1}\!(\k
P^\infty \!\times S^k\!)\! \cong \!KO^{-1}_{G_{\k}}(S^k)\! \xrightarrow{c^*}\!\!
KO^{-1}_{G_{\k}}(M^k)\! \xrightarrow{p^*}\!\!
KO^{-1}_{G_{\k}}\!(S(\zeta^{d_{\k}n})),
\end{equation*}  
we may twist by $h$ to obtain a new framing $h \bar{s}(\bar{t},\bar{a})$. 
\end{construction}

We have the following analogues of \cite[Lem. 1.14, Cor. 1.15]{Sto88}:

\begin{lemma}
We have
$$p_{-(n+1)}([S(\zeta^{d_{\k}n} \oplus \epsilon_{\k}^{d_{\k}}), h\bar{s}(\bar{t},\bar{a})]) = [S(\zeta^{d_{\k}n}), h\bar{s}(\bar{t},\bar{a})] \in \Omega^{\k}_{d_{\k}n,k-1}.$$
\end{lemma}

\begin{proof}
As in the proof of \cite[Lem. 1.14]{Sto88}, the bundle $S(\zeta)
\subset S(\zeta \oplus \epsilon_{\k}^{d_{\k}})$ is the zero set of a
$G_{\k}$-equivariant map $S(\zeta \oplus \epsilon_{\k}) \to
\r^{d_{\k},0}_{\k}$. Since the framings $h\bar{s}(\bar{t},\bar{a})$ on
$S(\zeta \oplus \epsilon)$ and $S(\zeta)$ are compatible, the result
follows from the geometric identification of $p$ in \eqref{Eqn:pGeom}.
\end{proof}
\begin{corollary}
We have
\begin{equation*}
[S(\zeta^{d_{\k}n}), h\bar{s}(\bar{t},\bar{a})] = [S(\r^{d_{\k}n,0}_{\k}) \times M^k, h(\bar{c}_n^{\k} \times \bar{a})] \in \Omega^{\k}_{d_{\k}n,k-1}.
\end{equation*}
\end{corollary}
\begin{proof}
As in the discussion preceding \cite[Cor. 1.15]{Sto88}, this follows
from iterated applications of the previous lemma, along with the
stable triviality of $\zeta$.
\end{proof}
\begin{construction}[{The map $\tilde{J}_{\k}$}]\label{Cnstr:JK}
Let $\bar{d}$ be the trivial framing of $S^k$. The assignment
$$h \mapsto [S(\r^{d_{\k}n,0}_{\k}) \times S(\r^{0,k+1}_{\k}), h(\bar{c}_n^{\k} \times \bar{d})] \in \Omega_{d_{\k}n,k-1}^{\k}$$
induces a well-defined map
\begin{equation*}
  \tilde{J}_{\k}: KO^{-1}(\k P^\infty \wedge S^k) \to
  \lim_n \Omega^{\k}_{d_{\k}n,k-1} \cong \lim_n \pi_{k-1}(\k P^\infty_{-n})
  \to \pi_{k+d_{\k}-1}^s.
\end{equation*}
\end{construction}
\begin{remark}
  The map
\begin{equation*}
  \tilde{J}: KO^{-1}(S^k) \to \pi_k^s
\end{equation*}
defined by Stolz in \cite[Eqn. 1.16]{Sto88} is obtained from the map
\begin{equation*}
\tilde{J}_{\r}: KO^{-1}(\r P^\infty \wedge S^k) \to \pi_k^s
\end{equation*}
defined above by precomposition with the $C_2$-transfer map
\begin{equation*}
\tr_{C_2}: \r P^\infty \to S^0.
\end{equation*}
We use $\tilde{J}_{\k}$ instead of its composite with the map induced
by the $G_{\k}$-transfer
\begin{equation*}
\tr_{G_{\k}}: \Sigma^{d_{\k}-1} \k P^\infty \to S^0
\end{equation*}
for the following reason. 
We recall that
Stolz proves \cite[Lem. 1.19]{Sto88} that
  the transfer map $\tr^*_{C_2}: KO^{-i}(S^0) \to
  KO^{-i}(\r P^\infty)$ is an isomorphism for all $i \in \z$, which
  allows him to freely identify elements in the source and target. The
  analogous isomorphism does not hold for $\tr^*_{U(1)}$ and
  $\tr^*_{Sp(1)}$; indeed, an elementary Atiyah--Hirzebruch spectral
  sequence calculation shows that the ranks of $KO^{-4i}(\c P^\infty)$
  and $KO^{-4i}(\h P^\infty)$ are greater than the rank of
  $KO^{-4i}(S^0)$ for all $i \in \z$. In any case, this means we
  cannot freely identify elements in the source and the target when
  $\k = \c$ and $\k = \h$.
\end{remark}
Modifying the proof of \cite[Lem. 1.17]{Sto88}, we obtain the
following, which implies that $\tilde{J}_{\k}$ is a homomorphism.
\begin{lemma}
Let $\zeta^{d_{\k}n}$ be a stably trivial bundle over a framed
manifold $M^k$, let $\bar{s}$ be a $(d_{\k}n,k-1)$-$\k$-framing of
$S(\zeta^{d_{\k}n})$, and let $h \in KO^{-1}(\k P^\infty \wedge
S^k)$. Then
\begin{equation*}
  [S(\zeta^{d_{\k}n}), h\bar{s}] = [S(\zeta^{d_{\k}n}), \bar{s}] + [S(\r^{d_{\k}n,0}_{\k}) \times S(\r^{0,k+1}_{\k}), h(\bar{c}_n^{\k} \times \bar{d})] \in \Omega_{d_{\k}n,k-1}^{\k}.
\end{equation*}
\end{lemma}
Putting everything together, we obtain the following analogue of
\cite[Prop. 1.18]{Sto88}:

\begin{proposition}\label{Prop:Stolz1.18}
Let $\Sigma^k$ and $\Sigma^{k+d_{\k}n}$ be homotopy spheres and let
$\xi^{d_{\k}n+d_{\k}}$ be a vector bundle over $\Sigma^{k}$ with
stable trivialization $\bar{t}$. The following are equivalent:
\begin{enumerate}
\item There exist framings $\bar{a}$ and $\bar{b}$ of $\Sigma^k$ and
  $\Sigma^{k+d_{\k}n}$, respectively, and an element $h \in KO^{-1}(\k
  P^\infty \wedge S^k)$ such that
\begin{enumerate}
\item $n \leq M$, where $M$ is the $\k$-Mahowald filtration of
  $[\Sigma^k,\bar{a}] + \tilde{J}_{\k}(h)$, and
\item $[\Sigma^{k+d_{\k}n},\bar{b}] \in M([\Sigma^k,\bar{a}] +
  \tilde{J}_{\k}(h))$ if $n=M$ and is trivial if $n<M$.
\end{enumerate}
\item There exist framings $\bar{a}$ and $\bar{b}$ of $\Sigma^k$ and
  $\Sigma^{k+d_{\k}n}$, respectively, and an element
  $h \in KO^{-1}(\k P^\infty \wedge S^k)$ such that
\begin{equation*}
         [S(\xi^{d_{\k}n+d_{\k}}), h\bar{s}(\bar{t},\bar{a})] =
         [G_{\k} \times \Sigma^{k+d_{\k}n}, G_{\k} \times \bar{b}] \in
         \Omega^{\k}_{d_{\k}(n+1),{k-1}}.
\end{equation*} 
\end{enumerate}
\end{proposition}

Finally, we state the following relation between the homomorphisms $\tilde{J}_{\k}$ and the classical $J$-homomorphism. We defer the proof until the end of \cref{Sec:MT}.

\begin{proposition}\label{prop:jj}
The image of $\tilde{J}_{\k}$ is contained in the image of the classical $J$-homomorphism.
\end{proposition}

\section{The main theorem}\label{Sec:MT}
We are now prepared to prove \cref{MT:MI}. The proof is essentially
the same in the real, complex, and quaternionic cases, so to begin, we
restate the theorem in a more concise form:
\begin{thm}\label{Thm:MI}
Let $(\Sigma^k, \Sigma^{k+d_{\k}n})$ be a pair of homotopy
spheres. Suppose there exist framings $\bar{a}$ and $\bar{b}$ for
$\Sigma^k$ and $\Sigma^{k+d_{\k}n}$, respectively, and an element $h
\in KO^{-1}(\k P^\infty \wedge S^k)$ such that
\begin{enumerate}
\item The codimension $d_{\k} n$ is bounded above by the $\k$-Mahowald filtration $M$ of $[\Sigma^k,\bar{a}] + \tilde{J}_{\k}(h)$, and
\item We have
$$[\Sigma^{k+d_{\k}n},\bar{b}] \in  \begin{cases}
M([\Sigma^k,\bar{a}] + \tilde{J}_{\k}(h)]) \quad & \text{ if } d_{\k}n=M, \\
0 \quad & \text{ if } d_{\k}n<M.
\end{cases}
$$
\end{enumerate}
Then if $\k = \c$ or $\k = \h$, there exists a smooth $G_{\k}$-action
on $\Sigma^{k+d_{\k}n} \# \Sigma'$, for some homotopy
$(k+d_{\k}n)$-sphere $\Sigma'$ which bounds a parallelizable manifold,
with fixed points $\Sigma^k$. If $\k = \r$, the same conclusion holds
provided that $n > k+1$ and either $n+k$ and $n$ are both odd or $n$
is even and $n+k \equiv 1 \mod 4$.
\end{thm}
\begin{remark}
The case $\k=\r$ is the main result of \cite{Sto88}. Our proof of the
more general theorem is a direct adaptation of Stolz's proof from
\cite[Sec. 3]{Sto88}.
\end{remark}

\begin{remark}\label{Rmk:RemoveJ}
\cref{MT:MI} follows immediately from \cref{Thm:MI} and \cref{prop:jj}; compare with \cite[Rmk. (iii)]{Sto88}.
\end{remark}

\begin{proof}[{Proof of \cref{Thm:MI}}]
Choose an embedding $\Sigma^k \hookrightarrow \Sigma^{k+d_{\k}n}$. By
\cref{Prop:Stolz1.18}, we have
\begin{equation*}
  [S(\zeta \oplus \epsilon^{d_{\k}}), h\bar{s}(\bar{t},\bar{a})] =
  [G_{\k} \times \Sigma^{k+d_{\k}n}, G_{\k} \times \bar{b}] \in
  \Omega^{\k}_{d_{\k}(n+1),{k-1}},
\end{equation*}
where the action of $G_{\k}$ on $S(\zeta
\oplus \epsilon^{d_{\k}})$ comes from the standard action of $G_{\k}$
on $\epsilon^{d_{\k}}=\k$, and its action on the product $G_{\k}
\times \Sigma^{k+d_{\k}n}$ comes from the action of $G_{\k}$ on 
itself. Let $(W,\bar{w})$ be a $G_{\k}$-equivariant framed
bordism between these manifolds.

Choose an equivariant map $s: W \to \r_{\k}^{1,0}$ transverse to $0
\in \r_{\k}^{1,0}$ such that $s|_{S(\zeta \oplus \epsilon^{d_{\k}})}$
is the map projecting to the last coordinate and $s(G_{\k} \times
\Sigma^{k+d_{\k}n})$ is contained in the unit sphere
  $G_{\k}\subset \r^{1,0}_{\k}$. Let $V = s^{-1}(0) \subset W$.

Then the manifold $W':=W/G_{\k}$ is a framed bordism between 
  the manifolds
\begin{equation*}
    S(\zeta \oplus \epsilon^{d_{\k}})/G_{\k}\cup V \ \ \ \mbox{and}  \ \
    \Sigma^{k+d_{\k}n}.
\end{equation*}
We notice that
\begin{equation*}
\begin{array}{rcl}
  S(\zeta \oplus \epsilon^{d_{\k}})/G_{\k} &\cong&\displaystyle
  \left((S(\zeta)\times
  D(\epsilon^{d_{\k}}))\cup_{(S(\zeta)\times S(\epsilon^{d_{\k}}))}
  (D(\zeta)\times S(\epsilon^{d_{\k}})) \right)/G_{\k}
  \\
  \\
  &\cong& \displaystyle
  (S(\zeta)\times I)\cup_{S(\zeta)\times \{1\}}(D(\zeta)\times \{1\})
  \\
  \\
  &\cong & \displaystyle D(\zeta),
\end{array}  
\end{equation*}  
where we contract $S(\zeta)\times I$ to the boundary of $\displaystyle
D(\zeta)$. Here, we emphasize that the group $G_{\k}$ acts on
  $D(\zeta)\cup_{S(\zeta)} V$ by rotating the fibers of $D(\zeta)$ (with
  fixed points $\Sigma^k\subset D(\zeta)$) and $G_{\k}$ acts freely on
  $V$ by restricting the action on $W$ which extends the free action
  (by rotating the fibers) on the sphere bundle $S(\zeta)$. In
  particular, we have that $(D(\zeta)\cup_{S(\zeta)}V)^{G_\k} =
  \Sigma^k$.

Our next goal is to make the bordism $W'$ into an
  $h$-cobordism through surgery, so then $D(\zeta) \cup V$ belongs to
  the same diffeomorphism class as $\Sigma^{k+d_{\k}n}$, and we obtain
  the desired $G_{\k}$-action with fixed points $\Sigma^k$ from the
  obvious action on $D(\zeta) \cup V$. Making $W'$ into an
  $h$-cobordism proceeds in two steps.

First, there is no way for $W'$ to be an
$h$-cobordism unless the manifolds $D(\zeta) \cup V$
and $\Sigma^{k+d_{\k}n}$ are homotopy equivalent, i.e., unless
$D(\zeta) \cup V$ is a homotopy sphere. Therefore our first step is to
modify $V$ by equivariant surgeries relative to the boundary to make
the inclusion $S^{n-1} \hookrightarrow S(\zeta) = \partial V$ into a
homotopy equivalence. This makes $D(\zeta) \cup V$ into a homotopy
sphere, which can be verified using the exact homology sequence
associated to the cofibration $V \to D(\zeta) \cup V \to
\Th(\zeta)$. The traces of the surgeries yield a new framed bordism
between the new $D(\zeta) \cup V$ and $\Sigma^{k+d_{\k}n}$, which we
will still denote by $W'$.

Second, now that $D(\zeta) \cup V$ and $\Sigma^{k+d_{\k}n}$ are both
homotopy spheres, we can appeal to classical results of Kervaire and
Milnor \cite{KM63} to make $W'$ into an
$h$-cobordism, provided we sum with an appropriate homotopy
$(k+d_{\k}n)$-sphere $\Sigma'$ which bounds a parallelizable manifold.


Thus, our task is to carry out the first step of modifying
{the manifold} $V$ so that $S^{n-1} \hookrightarrow
S(\zeta) = \partial V$ is a homotopy equivalence.

As $V = s^{-1}(0)$, $W$ is {$(d_{\k}(n+1),k)$}-framed,
and $s$ was transverse to $0 \in \r_{\k}^{1,0}$, $V$ is
{$(d_{\k}n,k)$}-$\k$-framed. Under the correspondence
between $(d_{\k}n,k)$-${\k}$-framed manifolds and normal bordism groups from
\cref{Prop:ProjBord}, the $(d_{\k}n,k)$-$\k$-framing on $V$ corresponds to a
$(-n\gamma_{\k})$-structure $(v,\bar{v})$ on $\bar{V} = V/G_{\k}$,
\[
\begin{tikzcd}
\nu(\bar{V}) \arrow{d} \arrow{r}{\bar{v}} & -n\gamma_{\k} \arrow{d} \\
\bar{V} \arrow{r} & \k P^\infty. 
\end{tikzcd}
\]
Decomposing the base space $\Sigma^k$ into $\Sigma^k = D_+ \cup_f
D_-$, we obtain a decomposition
\begin{equation*}
\partial \bar{V} = S(\zeta)/G_{\k} = \k P(\zeta|_{D_+}) \cup_g \k P(\zeta|_{D_-})
\end{equation*}
with $g: \k P(\zeta|_{S^{k-1}}) \xrightarrow{\cong} \k
P(\zeta|_{S^{k-1}})$ a diffeomorphism. Thus $\bar{V}$ is a relative
bordism between $\k P(\zeta|_{D_+})$ and $\k P(\zeta|_{D_-})$, and it
is a relative $h$-cobordism if and only if $S^{n-1} \hookrightarrow
S(\zeta)=\partial V$ is a homotopy equivalence.

Since $\zeta|_{D_{\pm}}$ is a trivial bundle, the restriction $v|: \k
P(\zeta|_{D_{\pm}}) \to \k P^\infty$ is a
$d_{\k}(n-1)$-equivalence. Therefore the triple $(\k P(\zeta|_{\pm}),
v|, \bar{v}|)$ is a normal $(d_{\k}n-d_{\k}-1)$-smoothing in the sense
of Kreck \cite[Pg. 711]{Kre99}, and we can appeal to \cite[Thms. 3 and
  4]{Kre99} to determine when $\bar{V}$ can be made into a relative
$h$-cobordism.

The obstruction to obtaining a relative $h$-cobordism is an element
$\theta(\bar{V},\bar{\nu})$ in a certain quotient of an
$L$-group. When $\k = \r$, the vanishing of this obstruction can be
guaranteed by imposing the conditions on $k$ and $n$ in the statement
of the theorem. When $\k = \c$ and $\k = \h$, these obstructions
automatically vanish: we have $B = \k P^\infty$, and thus $\pi_1(B) =
0 = w_1(B)$.
\end{proof}

\begin{proof}[Proof of \cref{prop:relations-Mahowald-invariants}]
{We denote by $\Gamma$ the group $G_{\c}$ or $G_{\h}$. Suppose that $[\Sigma^{k+d_{\k}n},\bar{b}] = M([\Sigma^k,\bar{a}])$, and let $D(\zeta)\cup V \cong \Sigma^{k+d_{\k}n}$ be as in the proof above equipped with the
  action of $\Gamma$.  Then, by construction, $(D(\zeta)\cup
  V)^{\Gamma}=\Sigma^k$. Since $C_2\subset \Gamma$, the construction
  above implies that $(D(\zeta)\cup V)^{C_2}=\Sigma^k$. Indeed, it is
  obvious that $D(\zeta)^{\Gamma}=D(\zeta)^{C_2}$ and the action of
  $C_2\subset \Gamma$ is free. According to \cite[Theorem D]{Sto88},
  there exist corresponding framings $\bar{a}'$ and $\bar b'$ on $\Sigma^k$
  and $\Sigma^{k+d_{\k}n}$ such that $M_{\r}[(\Sigma^k,\bar a')] =
  [(\Sigma^{k+d_{\k}n},\bar b')]$; however, the framings $\bar a'$ and $\bar b'$ may differ from the initial framings $\bar a$ and $\bar b$. Thus we have either $M_{\k}(\alpha) =
  M_{\r}(\alpha)$ or $M_{\k}(\alpha) = \alpha$, where the first equality holds only in $\coker(J)$.} 
\end{proof}

\begin{proof}[Proof of \cref{prop:jj}]
{ Let $j \in KO^{-1}(\k P^\infty \wedge S^k)$ and let $[\Sigma^k, \bar{\alpha}] = \tilde{J}_{\k}(j) \in \pi_k^s$. Let $[\Sigma^{k+d_{\k}n},\bar{b}] \in M_{\k}([\Sigma^k, \bar{\alpha}])$. By \cref{MT:MI}, there is a $G_{\k}$-action on $\Sigma^{k+d_{\k}n}$ such that $(\Sigma^{k+d_{\k}n})^{G_{\k}} = \Sigma^k$, and by the discussion from the previous proof, restricting to $C_2 \subset G_{\k}$ yields $(\Sigma^{k+d_{\k}n})^{C_2} = \Sigma^k$. Iterating the $\k$-Mahowald invariant, we may realize $\Sigma^k$ as the fixed points of a smooth $C_2$-action on a homology sphere of arbitrarily large dimension. \cite[Thm. C]{Sch82} then implies that $[\Sigma^k, \bar{a}]$  lies in the image of the ordinary $J$-homomorphism.
}
\end{proof}

\section{Applications}\label{Sec:Applications}
In this section, we compute some complex and quaternionic Mahowald
invariants (\cref{SS:CMI} and \cref{SS:HMI}) using the Adams spectral sequence. We record the applications to
transformation groups of homotopy spheres in \cref{SS:TG}.

\subsection{Complex Mahowald invariants}\label{SS:CMI}

We record the following facts about the homotopy groups of stunted complex projective spectra, each of which can be verified using the Adams spectral sequence:

\begin{enumerate}

\item In $\c P^\infty_{-2}$, we see that $\nu[-4] = \eta[-2]$, so $\nu \in M_{\c}(\eta)$. 

\item In $\c P^\infty_{-3}$, we see that $\nu^2[-6] = \eta^2[-2]$, so $\nu^2 \in M_{\c}(\eta^2)$. 

\item In $\c P^\infty_{-3}$, we see that $\sigma[-6] = \nu[-2]$, so $\sigma \in M_{\c}(\nu)$. 

\item In $\c P^\infty_{-4}$, we see that $\nu^3[-8] = \eta^3[-2]$, so $\nu^3 \in M_{\c}(\eta^3)$. 

\item In $\c P^\infty_{-5}$, we see that $\sigma^2[-10] = \nu^2[-2]$, so $\sigma^2 \in M_{\c}(\nu^2)$. 

\item In $\c P^\infty_{-7}$, we see that $\sigma^3[-14] = \nu^3[-2]$, so $\sigma^3 \in M_{\c}(\nu^3)$. 

\end{enumerate}

\begin{proposition}\label{Prop:CMIHopf}
We have
$$\nu \in M(\eta), \quad \nu^2 \in M(\eta^2), \quad \nu^3 \in M(\eta^3),$$
$$\sigma \in M(\nu), \quad \sigma^2 \in M(\nu^2), \quad \sigma^3 \in M(\nu^3).$$
\end{proposition}

\subsection{Quaternionic Mahowald invariants}\label{SS:HMI}

We record the following facts about the homotopy groups of stunted quaternionic projective spectra, each of which can be verified using the Adams spectral sequence:

\begin{enumerate}

\item In $\h P^\infty_{-2}$, we see that for $i = 0, 1, 2$, $2^i\sigma[-8] = 2^i\nu[-4]$, so $2^i \sigma \in M_{\h}(2^i \nu)$. 

\item In $\h P^\infty_{-3}$, we see that $\sigma^2[-12] = \nu^2[-4]$, so $\sigma^2 \in M_{\h}(\nu^2)$. 

\item In $\h P^\infty_{-4}$, we see that $\sigma^3[-16] = \nu^3[-4]$, so $\sigma^3 \in M_{\h}(\nu^3)$. 

\end{enumerate}

\begin{proposition}\label{Prop:HMIHopf}
We have
$$2^i\sigma \in M(2^i\nu), \quad \sigma^2 \in M(\nu^2), \quad \sigma^3 \in M(\nu^3),$$
where $i = 0,1,2$.  
\end{proposition}

\subsection{Consequences for transformation groups}\label{SS:TG}

Feeding the complex and quaternionic Mahowald invariant computations above into \cref{MT:MI}, we learn the following about smooth $U(1)$- and $Sp(1)$-actions on homotopy spheres:

\begin{corollary}[{Compare with \cref{Prop:CMIHopf} and \cref{Prop:HMIHopf}}]\label{Cor:Actions} \
\begin{enumerate}
\item The homotopy spheres $S^3$, $S^6$, and $S^9$ corresponding to $\nu$, $\nu^2$, and $\nu^3$ admit nontrivial $U(1)$-actions with fixed points the homotopy spheres $S^1$, $S^2$, and $S^3$ corresponding to $\eta$, $\eta^2$, and $\eta^3$, respectively. 

\item The homotopy spheres $S^7$ and $\Sigma^{21}$ corresponding to $\sigma$ and $\sigma^3$ admit nontrivial $U(1)$- and $Sp(1)$-actions with fixed points the homotopy spheres $S^3$ and $\Sigma^9$ corresponding to $\nu$ and $\nu^3$, respectively. 

\item The homotopy spheres $S^7$ and $\Sigma^{21}$ corresponding to $\sigma$ and $\sigma^3$ admit nontrivial $(U(1) \times Sp(1))$- and $U(1)^{\times 2}$-actions with fixed points the homotopy spheres $S^1$ and $S^3$ corresponding to $\eta$ and $\eta^3$, respectively. 
\end{enumerate}
\end{corollary}

\bibliographystyle{alpha}
\bibliography{master}

\end{document}